\documentclass[enlgish]{ourlematema}

\usepackage{aliascnt}
\usepackage[colorlinks=true,linkcolor=blue,urlcolor=blue,pagebackref]{hyperref}
\usepackage{xcolor, graphics, graphicx, mathbbol}
\usepackage{tikz}
\usepackage{enumerate}

\newcommand{\CC}{\mathbb{C}}

\newcommand{\NN}{\normalfont\mathbb{N}}

\newcommand{\PP}{{\normalfont\mathbb{P}}}

\newcommand{\nn}{{\normalfont\mathfrak{N}}}

\newcommand{\fJ}{{\mathfrak{J}}}

\newcommand{\Ker}{\normalfont\text{Ker}}

\newcommand{\Id}{{\normalfont\text{Id}}}

\newcommand{\Sym}{\normalfont\text{Sym}}
\newcommand{\Rees}{\mathcal{R}}

\newcommand{\SSS}{\mathbb{S}}
\newcommand{\EEQ}{\mathcal{K}}

\newcommand{\FF}{\normalfont\mathcal{F}}

\newcommand{\bideg}{\normalfont\text{bideg}}

\newcommand{\Hilb}{{\normalfont\text{Hilb}}}
\newcommand{\Spec}{\normalfont\text{Spec}}

\newcommand{\biProj}{{\normalfont\text{BiProj}}}

\def\f0{\mathbf{0}}

\newtheorem{theorem}{Theorem}[section]

\newtheorem{headthm}{Theorem}

\newaliascnt{headcor}{headthm}
\newtheorem{headcor}[headcor]{Corollary}
\aliascntresetthe{headcor}

\newaliascnt{headthmdef}{headthm}

\aliascntresetthe{headthmdef}

\newaliascnt{headconj}{headthm}

\aliascntresetthe{headconj}

\newaliascnt{corollary}{theorem}

\aliascntresetthe{corollary}

\newaliascnt{lem}{theorem}

\aliascntresetthe{lem}

\newaliascnt{conjecture}{theorem}

\aliascntresetthe{conjecture}

\newaliascnt{proposition}{theorem}
\newtheorem{proposition}[proposition]{Proposition}
\aliascntresetthe{proposition}

\theoremstyle{definition}
\newaliascnt{definition}{theorem}

\aliascntresetthe{definition}

\newaliascnt{notation}{theorem}

\aliascntresetthe{notation}

\newaliascnt{example}{theorem}

\aliascntresetthe{example}

\newaliascnt{examples}{theorem}

\aliascntresetthe{examples}

\newaliascnt{remark}{theorem}
\newtheorem{remark}[remark]{Remark}
\aliascntresetthe{remark}

\newaliascnt{problem}{theorem}

\aliascntresetthe{problem}

\newaliascnt{question}{theorem}

\aliascntresetthe{question}

\newaliascnt{convention}{theorem}

\aliascntresetthe{convention}

\newaliascnt{construction}{theorem}

\aliascntresetthe{construction}

\newaliascnt{setup}{theorem}
\newtheorem{setup}[setup]{Setup}
\aliascntresetthe{setup}

\newaliascnt{algorithm}{theorem}

\aliascntresetthe{algorithm}

\newaliascnt{observation}{theorem}

\aliascntresetthe{observation}

\newaliascnt{defprop}{theorem}

\aliascntresetthe{defprop}

\def\equationautorefname~#1\null{(#1)\null}
\def\sectionautorefname~#1\null{Section #1\null}
\def\subsectionautorefname~#1\null{\S #1\null}

\title{Equations and multidegrees for inverse symmetric matrix pairs}

\author{Yairon Cid-Ruiz}
\address{%
	Department of Mathematics: Algebra and Geometry, Ghent University, Krijgslaan 281 – S25, 9000 Ghent, Belgium\\
	\email{Yairon.CidRuiz@UGent.be}
}

\MSC{14C17, 13H15, 62R01, 13D02, 13A30.}
\keywords{symmetric matrix, multidegrees, maximum likelihood degree, rational map, Rees algebra, symmetric algebra.}

\begin{document}

\begin{abstract}
	\noindent
	We compute the equations and multidegrees of the biprojective variety that parametrizes pairs of symmetric matrices that are inverse to each other. 
	As a consequence of our work, we provide an alternative proof for a result of Manivel, Micha\l{}ek, Monin, Seynnaeve and Vodi\v{c}ka that settles a previous conjecture of Sturmfels and Uhler regarding the polynomiality of maximum likelihood degree.
\end{abstract}

\section{Introduction}
\label{sect_intro}

The purpose of this paper is to study the biprojective variety that parametrizes pairs of symmetric matrices that are inverse to each other. 
Let $\SSS^n$ be the space of symmetric $n \times n$ matrices over the complex numbers $\CC$.
Let $\PP^{m-1}$ be the projectivization $\PP^{m-1} = \PP(\SSS^n)$ of $\SSS^n$, where $m = \binom{n+1}{2}$.
We are interested in the biprojective variety $\Gamma \subset \PP^{m-1} \times \PP^{m-1}$ given as follows
$$
\Gamma \,:= \, \overline{\big\lbrace (M, M^{-1}) \mid M \in \PP(\SSS^n) \text{ and } \det(M) \neq 0  \big\rbrace} \,\subset\, \PP^{m-1} \times \PP^{m-1};
$$
i.e., the closure of all possible pairs of an invertible symmetric matrix and its inverse.

\medskip

Our main results are determining the equations and multidegrees of the biprojective variety $\Gamma$.
Before presenting them, we establish some notation.
Let $X={\left(X_{i,j}\right)}_{1 \le i,j \le n}$ and $Y={\left(Y_{i,j}\right)}_{1 \le i,j \le n}$ be generic symmetric matrices; i.e., $X_{i,j}$ and $Y_{i,j}$ are new variables over $\CC$.
Let $R$ be the standard graded polynomial ring $R = \CC[X_{i,j}]$, and $S$ be the standard bigraded polynomial ring $S = \CC[X_{i,j}, Y_{i,j}]$ where $\bideg(X_{i,j})=(1,0)$ and $\bideg(Y_{i,j})=(0,1)$.

Let $\fJ \subset S$ be ideal of the defining equations of $\Gamma$.

As $\dim(\Gamma) = m - 1$, for each $i, j \in \NN$ with $i+j=m-1$, one considers the \emph{multidegree $\deg^{i,j}(\Gamma)$ of $\Gamma$ of type $(i,j)$}.
Geometrically, $\deg^{i,j}(\Gamma)$ equals the number of points in the intersection of $\Gamma$ with the product $L \times M  \subset \PP^{m-1} \times \PP^{m-1}$, where $L \subset \PP^{m-1}$ and $M \subset \PP^{m-1}$ are general linear subspaces of dimension $m-1-i$ and $m-1-j$, respectively.
Following the notation of \cite[\S 8.5]{MILLER_STURMFELS}, we say that the \emph{multidegree polynomial} of $\Gamma$ is given by 
$$
\mathcal{C}(\Gamma;t_1,t_2) \;:=\; \sum_{i+j=m-1} \deg^{i,j}(\Gamma) \, t_1^{m-1-i}t_2^{m-1-j} \; \in \; \NN[t_1,t_2]
$$
(also, see \cite[Theorem A]{MIXED_MULT}, \cite[Remark 2.9]{POTIVITY}).

A fundamental idea in our approach is to reduce the study of $\Gamma$ to instead considering the biprojective variety of pairs of symmetric matrices with product zero.
Let $\Sigma \subset \PP^{m-1} \times \PP^{m-1}$ be the biprojective variety parametrized by pairs of symmetric matrices with product zero; i.e., by pairs of symmetric matrices $(M,N) \in \PP^{m-1} \times \PP^{m-1}$ such that $MN=0$.
The ideal of defining equations of $\Sigma$ is clearly given by
$$
I_1(XY),
$$
where $I_1(XY)$ denotes the ideal generated by the $1\times1$-minors (i.e, the entries) of the matrix $XY$.
Similarly, since $\dim(\Sigma) = m-2$, we define the multidegree polynomial 
$$
\mathcal{C}(\Sigma;t_1,t_2) \;:=\; \sum_{i+j=m-2} \deg^{i,j}(\Sigma) \, t_1^{m-1-i}t_2^{m-1-j} \; \in \; \NN[t_1,t_2]
$$
 of $\Sigma$.

\medskip 

The theorem below provides the defining equations of $\Gamma$.
It also shows that the study of $\mathcal{C}(\Gamma;t_1,t_2)$ can be substituted to considering $\mathcal{C}(\Sigma;t_1,t_2)$ instead.
Our proof depends on translating our questions in terms of Rees algebras and on using the results of Kotsev \cite{KOTZEV}.

\begin{headthm}
	\label{thmA}
	Under the above notations, the following statements hold: 
	\begin{enumerate}[(i)]
	\item  $\fJ$ is a prime ideal given by 
		$$
		\fJ \,=\, I_1\left(XY - b\Id_n\right) \,=\, \left(
		\begin{array}{ll}
			\sum_{k=1}^nX_{i,k}Y_{k,j}, & 1 \le i \ne j \le n\\
			\sum_{k=1}^nX_{i,k}Y_{k,i} - 	 \sum_{k=1}^nX_{j,k}Y_{k,j}, & 1 \le i, j \le n
		\end{array}	 
		\right)
		$$
		where $b = (XY)_{1,1} = \sum_{l=1}^nX_{1,k}Y_{k,1} \in S$ and $\Id_n$ denotes the $n\times n$ identity matrix.
	\item We have the following equality relating multidegree polynomials
		$$
		t_1^m + t_2^m + \mathcal{C}(\Sigma;t_1,t_2) \; = \; (t_1+t_2) \cdot \mathcal{C}(\Gamma;t_1,t_2). 
		$$
	\end{enumerate}
\end{headthm}

Our second main result is obtaining general formulas for the multidegrees of $\Gamma$ and $\Sigma$.
Here our approach depends on previous computations that were made by Nie, Ranestad and Sturmfels \cite{NIE_RANESTAD_STURMFELS}, and by von Bothmer and Ranestad \cite{BOTHMER_RANESTAD}.
The formula we obtained is expressed in terms of a function on subsequences of $\{1,\ldots,n\}$.
Let 
$$
\psi_i = 2^{i-1}, \quad \psi_{i,j} = \sum_{k=i}^{j-1}\binom{i+j-2}{k} \quad \text{when } i < j,
$$
and for any $\alpha = (\alpha_1,\ldots,\alpha_r) \subset \{ 1, \ldots, n\}$ let
$$
\psi_{\alpha} = 
\begin{cases}
	\text{Pf}\left(\psi_{\alpha_k, \alpha_l}\right)_{1 \le k < l \le n} \quad \text{ if } r \text{ is even},\\
	\text{Pf}\left(\psi_{\alpha_k, \alpha_l}\right)_{0 \le k < l \le n} \quad \text{ if } r \text{ is odd},
\end{cases}
$$
where $\psi_{\alpha_0,\alpha_k} = \psi_{\alpha_k}$ and $\text{Pf}$ denotes the Pfaffian.
For any $\alpha \subset \{ 1, \ldots, n\}$, the complement $\{1,\ldots,n\} \setminus \alpha$ is denoted by $\alpha^c$.
By an abuse of notation we set $\psi_{\emptyset} = 1$.

\begin{headthm}\label{thmB}
	Under the above notations, the following statements hold: \begin{enumerate}[(i)]
		\item The multidegree polynomial of $\Sigma$ is determined by the equation
		$$
		t_1^m + t_2^m + \mathcal{C}(\Sigma;t_1,t_2) \,=\, \sum_{d=0}^{m} 
		 \beta(n,d)\,t_1^{m-d}t_2^d,
		$$
		where 
		$$
		\beta(n,d) \,:=\, \sum_{\substack{\alpha \subset \{ 1, \ldots, n \}\\ ||\alpha|| = d} 
		} \psi_\alpha \psi_{\alpha^c};
		$$
		 in the last sum $\alpha$ runs over all strictly increasing subsequences of $\{1,\ldots,n\}$, including the case $\alpha = \emptyset$, and $||\alpha||$ denotes the sum of the entries of $\alpha$.
		 
		 \item For each $0 \le d \le m-1$, we have the equality 
		 $$
		 \deg^{m-1-d,d}(\Gamma) \,=\, \sum_{j=0}^{d} {(-1)}^j \beta(n,d-j).
		 $$
	\end{enumerate}
\end{headthm}

Our last interest is on the maximum likelihood degree (ML-degree) of the general linear concentration model (see \cite{STURMFELS_UHLER}, \cite{michalek2020maximum} for more details).
Let $\mathcal{L}$ be a general linear subspace of dimension $d$ in $\SSS^n$, and denote by $\mathcal{L}^{-1}$ the $(d-1)$-dimensional projective subvariety of $\PP^{m-1} = \PP(\SSS^n)$ obtained by inverting the matrices in $\mathcal{L}$.
From \cite[Theorem 1]{STURMFELS_UHLER}, the ML-degree of the general linear concentration model, denoted as $\phi(n,d)$, is equal to the degree of the projective variety $\mathcal{L}^{-1}$.
From the way $\Gamma$ is defined, it then follows that 
\begin{equation}
	\label{eq_phi_as_multdeg}
	\phi(n,d) \, = \, \deg^{m-d, d-1}(\Gamma).
\end{equation}
So, the computation of the invariants $\phi(n,d)$ can be reduced to determining the multidegrees of $\Gamma$ (which we did in \autoref{thmB}).

Finally, by using \autoref{thmB} and a result of Manivel, Micha\l{}ek, Monin, Seynnaeve and Vodi\v{c}ka regarding the polynomiality in $n$ of the function $\psi_{\{1,\ldots,n\} \setminus \alpha}$ (see \autoref{thm_poly_phi_c}), we obtain an alternative proof to a previous conjecture of Sturmfels and Uhler (see \cite[p. 611]{STURMFELS_UHLER}).

\begin{headcor}[{Manivel-Micha\l{}ek-Monin-Seynnaeve-Vodi\v{c}ka; \cite[Theorem 1.3]{Michalek2020CQ}}]
	\label{head_cor}
	For each $d \ge 1$, the function $\phi(n,d)$ coincides with a polynomial of degree $d-1$ in $n$.
\end{headcor}

The basic outline of this paper is as follows.
In \hyperref[sect_def_eq]{Section 2}, we compute the defining equations of $\Gamma$.
In \hyperref[sect_multded]{Section 3}, we determine the multidegrees of $\Gamma$ and $\Sigma$.
In \hyperref[sect_poly]{Section 4}, we show the polynomiality of $\phi(n,d)$.

\section{The defining equations of $\Gamma$}
\label{sect_def_eq}

During this section, we compute the defining equations of the variety $\Gamma$.
The following setup is used throughout the rest of this paper.

\begin{setup}
	Let $X={\left(X_{i,j}\right)}_{1\le i,j \le n}$ and $Y={\left(Y_{i,j}\right)}_{1\le i,j \le n}$ be generic symmetric matrices over $\CC$.
	Let $R$ be the standard graded polynomial ring $R = \CC[X_{i,j}]$, and $S$ be the standard bigraded polynomial ring $S = \CC[X_{i,j}, Y_{i,j}]$ where $\bideg(X_{i,j})=(1,0)$ and $\bideg(Y_{i,j})=(0,1)$.	
	Let $I=I_{n-1}(X)$ be the ideal of $(n-1)\times(n-1)$-minors of $X$.
	Let $t$ be a new indeterminate.
	The Rees algebra $\Rees(I):=\bigoplus_{n=0}^\infty I^nt^n \subset R[t]$ of $I$ can be presented as a quotient of $S$ by using the map
	\begin{eqnarray*}
		\label{presentation_Rees}
		\Psi\;:\; S & \longrightarrow & \Rees(I) \subset R[t] \\ \nonumber
		Y_{i,j} & \mapsto & Z_{i,j}t,
	\end{eqnarray*}
	where $Z_{i,j} \in R$ is the signed minor obtained by deleting $i$-th row and the $j$-th column.
	We set $\bideg(t)=(-n+1, 1)$, which implies that $\Psi$ is bihomogeneous of degree zero, and so $\Rees(I)$ has a natural structure of bigraded $S$-algebra.	
\end{setup}

Our point of departure comes from the following simple remarks.

\begin{remark}
	\label{rem_graph}
	For any matrix $M \in \SSS^n$, we denote its adjoint matrix as $M^+$.
	For any $M \in \SSS^n$ with $\det(M) \neq 0$, since $M^{-1} = \frac{1}{\det(M)} M^+$, it follows that $M^{-1}$ and $M^+$ represent the same point in $\PP^{m-1} = \PP(\SSS^n)$.
	Thus, we have that $\Gamma$ can be equivalently described as 
	$$
	\Gamma \,= \, \overline{\big\lbrace (M, M^{+}) \mid M \in \PP(\SSS^n) \text{ and } \det(M) \neq 0  \big\rbrace} \,\subset\, \PP^{m-1} \times \PP^{m-1}.
	$$
	Denote by $\FF : \PP^{m-1} \dashrightarrow \PP^{m-1}$ the rational map determined by signed minors $Z_{i,j}$, that is, 
	$$
	\FF : \PP^{m-1} \dashrightarrow \PP^{m-1}, \quad (X_{1,1} : X_{1,2} : \cdots :X_{n,n}) \mapsto (Z_{1,1} : Z_{1,2} : \cdots :Z_{n,n}).
	$$
	Therefore, we obtain that $\Gamma$ coincides with 
	$$
	\Gamma \, = \; \overline{\text{graph}(\FF)} \,\subset \, \PP^{m-1} \times \PP^{m-1},
	$$
	the closure of the graph of the rational map $\FF$.
\end{remark}

\begin{remark}
	\label{rem_eq_Rees}
	Notice that $I = I_{n-1}(X)$ by construction is the base ideal of the rational map $\FF$ -- the ideal generated by a linear system defining the rational map.
	So, it is a basic result that the Rees algebra $\Rees(I)$ coincides with the bihomogeneous coordinate ring of the closure of the graph of $\FF$.
	By \autoref{rem_graph}, the bihomogeneous coordinate ring of $\Gamma$ is given by the Rees algebra $\Rees(I)$.
	Hence, in  geometrical terms, we have the identification
	$$
	\Gamma \,=\, \biProj(\Rees(I))  \subset \PP^{m-1} \times \PP^{m-1}.
	$$
	In more algebraic terms: the ideal $\fJ \subset S$ considered in the \hyperref[sect_intro]{Introduction} coincides with the defining equations of the Rees algebra, that is, $\fJ = \Ker(\Psi)$.
	For the relations between rational maps and Rees algebras, see, e.g., \cite[Section 3]{SPECIALIZATION_RAT_MAPS}.
\end{remark}

In general the Rees algebra is a very difficult object to study, but, under the present conditions we shall see that it coincides with the symmetric algebra $\Sym(I)$ of $I$ (i.e., the ideal $I$ is of linear type).
So, the main idea is to bypass the Rees algebra and consider the symmetric algebra instead.

From a graded presentation of $I$
$$
F_1 \xrightarrow{\varphi} F_0 \xrightarrow{(Z_{1,1}, Z_{1,2}, \ldots, Z_{n,n})} I \rightarrow 0,
$$
the symmetric algebra $\Sym(I)$ automatically gets the presentation
\begin{equation}
	\label{eq_pres_Sym}
	\Sym(I) \,\cong\, S / I_1\big(\left[Y_{i,j}\right]\cdot\varphi\big)
\end{equation}
and obtains a natural structure of bigraded $S$-algebra (for more details on the symmetric algebra, see, e.g., \cite[\S A2.3]{EISEN_COMM}).
In general, we have a canonical exact sequence of bigraded $S$-modules relating both algebras 
\begin{equation*}
	0 \rightarrow \EEQ \rightarrow \Sym(I) \rightarrow \Rees(I) \rightarrow 0,		
\end{equation*}
where $\EEQ$ equals the $R$-torsion of $\Sym(I)$ (see \cite{MICALI_REES}).
However, in the present case, we shall see that $\Sym(I) = \Rees(I)$.

We are now ready to compute the defining equations of $\Gamma$.

\begin{proof}[Proof of \autoref{thmA} \rm(i)]
	Due to \autoref{rem_eq_Rees}, it suffices to compute the defining equations of the Rees algebra $\Rees(I)$.
	From \cite[Theorem A]{KOTZEV} we have that $I$ is of linear type, i.e., the canonical map
	$$
	\Sym(I) \twoheadrightarrow \Rees(I)
	$$
	is an isomorphism.
	So, $\fJ$ coincides with the ideal of defining equations of $\Sym(I)$.
	By using \cite{JOSEFIAK} or \cite{GOTO_COMPLEX} we obtain an explicit $R$-free resolution for the ideal $I$ which is of the form $0 \rightarrow J_3 \rightarrow J_2 \xrightarrow{\varphi} J_1 \rightarrow R \rightarrow R/I \rightarrow 0$.
	From the presentation $\varphi$ of $I$, we obtain the ideal 
	$$
	\fJ =  I_1\big(\left[Y_{i,j}\right]\cdot\varphi\big)
	$$
	of defining equations of the symmetric algebra (see (\ref{eq_pres_Sym}) above).
	Therefore, $\Rees(I) = \Sym(I)$ is a bigraded $S$-algebra presented by the quotient
	$$
	\Rees(I) = \Sym(I) \cong S/\fJ,
	$$
	and from the description of $\varphi$ (the syzygies of $I$) given in  \cite{JOSEFIAK} or \cite{GOTO_COMPLEX} we obtain
	 $$
	 \fJ = \left(
	 \begin{array}{ll}
	 \sum_{k=1}^nX_{i,k}Y_{k,j}, & 1 \le i \ne j \le n\\
	 \sum_{k=1}^nX_{i,k}Y_{k,i} - 	 \sum_{k=1}^nX_{j,k}Y_{k,j}, & 1 \le i, j \le n
	 \end{array}	 
	 \right). 
	 $$
	 Finally, it is clear that $\fJ = I_1\left(XY - b\Id_n\right)$.
\end{proof}

\section{Computation of the multidegrees of $\Gamma$}
\label{sect_multded}

In this section, we concentrate on computing the multidegrees of $\Gamma$.
The idea is to reduce this computation to instead compute the multidegrees of $\Sigma$ and then to use previous results obtained in \cite{NIE_RANESTAD_STURMFELS} and \cite{BOTHMER_RANESTAD}.

For each $r_1, r_2 \in \NN$, we define the following ideal 
$$
J(r_1,r_2) \, := \; I_1(XY) + I_{r_1+1}(X) + I_{r_2+1}(Y) \,\subset\, S. 
$$
The following proposition yields a primary decomposition of the ideal $I_1(XY)$ in terms of the ideals $J(r_1,r_2)$.
Its proof is easily obtained by using results from \cite{KOTZEV}.
Similarly, the properties of $I_1(XY)$ described below are known in a more geometric language (see \cite[Proposition 16]{BERND_BAD_PROJECTIONS} and the references given therein).

\begin{proposition}
	\label{prop_decomp}
	The following statements hold:
	\begin{enumerate}[(i)]
		\item If $r_1+r_2 \le n$, then $J(r_1,r_2)$ is a prime ideal.
		\item The ideal $I_1(XY)$ is equidimensional of dimension $m=\binom{n+1}{2}$ and radical with primary decomposition 
		$$
		I_1(XY) \, = \; \bigcap_{r = 0}^n J(r, n-r).
		$$
	\end{enumerate}
\end{proposition}
\begin{proof}
	(i) From \cite[Proposition 4.5]{KOTZEV}, we have that $B(r_1,r_2) = S / J(r_1,r_2)$ is a domain, so the result is clear.
	
	(ii) By \cite[Lemma 4.6]{KOTZEV}, we know that the canonical map 
	$$
	\frac{S}{I_1(XY)} \; \rightarrow \;  \prod_{r=0}^n \frac{S}{J(r,n-r)}
	$$
	is injective. 
	So, it is clear that $I_1(XY)  =  \bigcap_{r = 0}^n J(r, n-r)$.
	The dimension of the Rees algebra is equal to $\dim(\Rees(I)) = \dim(R) +1 = m+1$ (see, e.g., \cite[Theorem 5.1.4]{huneke2006integral}).
	By (\ref{eq_Rees_I_1}), $S/I_1(XY) \cong \Rees(I) / b \Rees(I)$, and so Krull's Principal Ideal Theorem (see, e.g., \cite[Theorem 13.5]{MATSUMURA}) yields that 
	$$
	\dim\left(S / J(r,n-r)\right) \,= \, \dim\left(\Rees(I)\right) - 1 = m
	$$
	for each $0 \le r \le n$.
	Therefore, the result follows.
\end{proof}

We now recall how to define the multidegree polynomial $\mathcal{C}(\Rees(I);t_1,t_2)$ of $\Rees(I)$ by using the Hilbert series of $\Rees(I)$ (see \cite[\S 8.5]{MILLER_STURMFELS}).
We can write the Hilbert series $$
\Hilb_{\Rees(I)}(t_1,t_2) \,:=\, \sum_{v_1,v_2\in \NN}\dim_\CC\big([\Rees(I)]_{v_1,v_2}\big) t_1^{v_1}t_2^{v_2} \,\in\, \NN[[t_1,t_2]]
$$ 
in the following way
$$
\Hilb_{\Rees(I)}(t_1,t_2) = \frac{K\left(\Rees(I);t_1,t_2\right)}{(1-t_1)^m(1-t_2)^m},
$$
where $K\left(\Rees(I);t_1,t_2\right)$ is called the $K$-polynomial of $\Rees(I)$ (for instance, by just computing a bigraded free $S$-resolution of $\Rees(I)$).
Then, we define 
$$
\mathcal{C}(\Rees(I);t_1,t_2) \,:=\, \text{sum of the terms of }\;\; K(\Rees(I); 1-t_1,1-t_2)\;\; \text{ of degree } = m-1.
$$
Additionally, we remark that $m-1$ is the minimal degree of the terms of 
$$
K\left(\Rees(I);1-t_1,1-t_2\right).
$$
In a similar way, we define the multidegree polynomials 
$$
\mathcal{C}\big(S/I_1(XY)\big) \quad \text{ and } \quad \mathcal{C}\big(S/J(r,n-r)\big)
$$
for each $0 \le r \le n$.

The multidegrees of the particular cases $S/J(0,n)$ and $S/J(n,0)$ are easily handled by the following remark.

\begin{remark}
	\label{rem_multdeg_special}
	Since $J(0,n) = I_1(X) = \left(X_{i,j}\right)$ and $J(n,0) = I_1(Y) = \left(Y_{i,j}\right)$, it follows from the definition of multidegrees that 
	$$
	\mathcal{C}\big(S/J(0,n);t_1,t_2\big) = t_1^{m} \quad\text{ and }\quad \mathcal{C}\big(S/J(n,0);t_1,t_2\big) = t_2^{m}.
	$$
\end{remark}

For notational purposes, we denote by $\nn := \left(X_{i,j}\right) \cap \left(Y_{i,j}\right) \subset S$ the irrelevant ideal in the current biprojective setting. 
We have the following equivalent descriptions of $\Gamma$ and $\Sigma$ in terms of the $\biProj$ construction
\begin{equation}
	\label{eq_biProj_Rees}
	\Gamma = \biProj(\Rees(I)) = \big\lbrace P \in \Spec(\Rees(I))  \mid P \text{ is bihomogeneous and } P \not\supseteq \nn\Rees(I) \big\rbrace
\end{equation}
and 
\begin{equation}
	\label{eq_biProj_I_1}
	\Sigma = \biProj\left(T\right) = \big\lbrace P \in \Spec\left(T\right)   \mid P \text{ is bihomogeneous and } P \not\supseteq \nn T \big\rbrace,
\end{equation}
where $T = S/I_1(XY)$.
For more details on the $\biProj$ construction, the reader is referred to \cite[\S 1]{HYRY_MULTIGRAD}.

Next, we have a remark showing that the multidegree polynomial of $\Gamma$ as introduced before coincides with the multidegree polynomial of the Rees algebra $\Rees(I)$.

\begin{remark}
	\label{rem_multideg_Gamma}
	Due to (\ref{eq_biProj_Rees}), the fact that $\left(0:_{\Rees(I)} \nn^{\infty}\right)=0$ and \cite[Remark 2.9]{POTIVITY}, it follows that 
	$$
	\mathcal{C}\left(\Gamma;t_1,t_2\right) \, = \, \mathcal{C}\left(\Rees(I);t_1,t_2\right).
	$$
\end{remark}

On the other hand, the following remark shows that the multidegree polynomials of $\Sigma$ and $S/I_1(XY)$ do not agree. 
Indeed, the minimal primes $J(0,n)$ and $J(n,0)$ of $I_1(XY)$ are irrelevant from a geometric point of view, and so they are taken into account in the multidegree polynomial of $I_1(XY)$ but not in the one of $\Sigma$.

\begin{remark}
	\label{rem_multdeg_sigma}
	For ease of notation, set $T = S/I_1(XY)$.
	Directly from (\ref{eq_biProj_I_1}), we get that
	$$
	\Sigma = \biProj(T) = \biProj\left(\frac{T}{\left(0:_T\nn^\infty\right)}\right) = \biProj\left(\frac{S}{\bigcap_{r = 1}^{n-1} J(r, n-r)}\right). 
	$$
	Let $T' = S/ \bigcap_{r = 1}^{n-1} J(r, n-r)$.
	Thus, since $\left(0:_{T'}\nn^\infty\right)=0$, \cite[Remark 2.9]{POTIVITY} yields the  equality 
	$
	\mathcal{C}\left(\Sigma; t_1,t_2\right) \, = \, \mathcal{C}\left(T'; t_1,t_2\right).
	$	
	Therefore, from \autoref{prop_decomp}, \autoref{rem_multdeg_special} and the additivity of multidegrees (see \cite[Theorem 8.53]{MILLER_STURMFELS}) we obtain the equality 
	$$
	\mathcal{C}\left(S/I_1(XY); t_1,t_2\right) \, = \, 
		t_1^m + t_2^m + \mathcal{C}\left(\Sigma; t_1,t_2\right). 
	$$
\end{remark}

The next result provides an important relation between the multidegrees of $\Gamma$ and $\Sigma$.

\begin{proof}[Proof of \autoref{thmA} \rm (ii)]
	First, we note the following trivial equality 
	$$
	\fJ + bS = I_1\left(XY - b\Id_n\right) + bS = I_1(XY).
	$$
	As $\Rees(I) \cong S/\fJ$ is clearly a domain and $\bideg(b)=(1,1)$, we obtain the short exact sequence 
	\begin{equation}
			\label{eq_Rees_I_1}
			0 \rightarrow \Rees(I)(-1,-1) \xrightarrow{\cdot b} \Rees(I) \rightarrow S/I_1(XY) \rightarrow 0
	\end{equation}
	and that $\dim(S/I_1(XY))=\dim(\Rees(I))-1$.
	Consequently, we get the following equality relating Hilbert series 
	$$
	\Hilb_{S/I_1(XY)}(t_1,t_2) = (1-t_1t_2) \cdot \Hilb_{\Rees(I)}(t_1,t_2).
	$$
	It then follows that $
	K(S/I_1(XY);t_1,t_2) = (1-t_1t_2) \cdot K(\Rees(I);t_1,t_2)$, and the substitutions $t_1\mapsto 1-t_1, t_2\mapsto 1-t_2$ yield the equation 
	$$
	K(S/I_1(XY);1-t_1,1-t_2) = (t_1+t_2-t_1t_2)\cdot K(\Rees(I);1-t_1,1-t_2).
	$$
	By choosing the terms of minimal degree in both sides of the last equation, we obtain
	$$
	\mathcal{C}(S/I_1(XY);t_1,t_2) \; = \; (t_1+t_2) \cdot \mathcal{C}(\Rees(I);t_1,t_2),
	$$
	and so the result follows \autoref{rem_multideg_Gamma} and \autoref{rem_multdeg_sigma}.
\end{proof}

In \cite{NIE_RANESTAD_STURMFELS} it was introduced the notion of algebraic degree of semidefinite programming.
By using \cite[Theorem 10]{NIE_RANESTAD_STURMFELS}, these invariants can be seen as the multidegrees of  $S / J(r,n-r)$ for $0 < r < n$.

\begin{theorem}[{Nie - Ranestad - Sturmfels; \cite[Theorem 10]{NIE_RANESTAD_STURMFELS}}]
	\label{thm_Nie_Ranestad_Sturmfels}
	For $0 < r < n$, we have that 
	$$
	\mathcal{C}\big(S/J(n-r,r);t_1,t_2\big) \, = \,  \sum_{d = 0}^{m} \delta(d,n,r)\, t_1^{m-d}t_2^d,
	$$
	where $\delta(d,n,r)$ denotes the algebraic degree of semidefinite programming.
\end{theorem}

We now present the following explicit formula for the algebraic degree of semidefinite programming that was obtained in \cite{BOTHMER_RANESTAD}.

\begin{theorem}[{von Bothmer - Ranestad; \cite[Theorem 1.1]{BOTHMER_RANESTAD}}]
	\label{thm_Bothmer_Ranestad}
	The algebraic degree of semidefinite programming is equal to
	$$
	\delta(d,n,r) \, = \,  \sum_{\alpha} \psi_\alpha \psi_{\alpha^c}, 
	$$
	where the sum runs over all strictly increasing subsequences $\alpha = \{ \alpha_1,\ldots, \alpha_{n-r} \}$ of $\{1 ,\ldots, n\}$ of length $n-r$ and sum $\alpha_1 + \cdots + \alpha_r = d$, and $\alpha^c$ is the complement $\{ 1, \ldots, n \} \setminus \alpha$.
\end{theorem}

After the previous discussions, we can now compute the multidegrees of $\Gamma$.

\begin{proof}[Proof of \autoref{thmB}]
	(i) First, we concentrate on computing the multidegrees of $S/I_1(XY)$.
	By using the additivity of multidegrees (see \cite[Theorem 8.53]{MILLER_STURMFELS}) together with \autoref{prop_decomp}, we obtain the following equality 
	$$
	\mathcal{C}\big(S/I_1(XY)\big) \,=\, \sum_{r=0}^n \mathcal{C}\big(S/J(r,n-r); t_1, t_2\big).
	$$
	Hence, by combining \autoref{thm_Nie_Ranestad_Sturmfels}, \autoref{rem_multdeg_special}  and \autoref{thm_Bothmer_Ranestad} it follows that 
	\begin{align*}
		\mathcal{C}\big(S/I_1(XY); t_1,t_2\big) &= t_1^m + t_2^m + \sum_{r=1}^{n-1}\sum_{d=0}^m \delta(d,n,r)\, t_1^{m-d}t_2^d\\
		&= 
		\sum_{d=0}^{m} \left(\sum_{\substack{\alpha \subset \{ 1, \ldots, n \}\\ ||\alpha|| = d} 
		} \psi_\alpha \psi_{\alpha^c}\right)\,t_1^{m-d}t_2^d,
	\end{align*}
	where in the last equation $\alpha$ runs over all strictly increasing subsequences of $\{1,\ldots,n\}$, including the case $\alpha = \emptyset$, and $||\alpha||$ denotes the sum of the entries of $\alpha$.
	Notice that $\psi_{\{1,\ldots,n\}} = 1$ (see \cite[Proposition A.15]{Laksov_Lascoux_Thorup}) and that by an abuse of notation we are setting $\psi_{\emptyset}=1$.
	Finally, by setting 
	$$
	\beta(n,d) \,=\, \sum_{\substack{\alpha \subset \{ 1, \ldots, n \}\\ ||\alpha|| = d} 
	} \psi_\alpha \psi_{\alpha^c},
	$$
	the result of this part follows from \autoref{rem_multdeg_sigma}.

	(ii) 
	Notice that \autoref{thmA} (ii) yields the equation 
	$$
	\beta(n,d) \,=\, \deg^{d,m-1-d}(\Gamma) + \deg^{d-1,m-d}(\Gamma).
	$$
	Since the ideal $\fJ$ of defining equations of $\Rees(I)$ is symmetric under swapping the variables $X_{i,j}$ and $Y_{i,j}$, it follows that $\deg^{d,m-1-d}(\Gamma) = \deg^{m-1-d, d}(\Gamma)$ for all $0 \le d \le m-1$.
	Accordingly, we have the equality 
	$$
	\beta(n,d) \,=\, \deg^{m-1-d,d}(\Gamma) + \deg^{m-d,d-1}(\Gamma).
	$$
	Therefore, the equation $\deg^{m-1-d,d}(\Gamma) \,=\, \sum_{j=0}^{d} {(-1)}^j \beta(n,d-j)$ is obtained iteratively.
\end{proof}

\section{Polynomiality of ML-degree}
\label{sect_poly}

During this short section, we show \autoref{head_cor}.
Our proof is an easy consequence of \autoref{thmB} and the following result.

\begin{theorem}[{Manivel-Micha\l{}ek-Monin-Seynnaeve-Vodi\v{c}ka; \cite[Theorem 4.3]{Michalek2020CQ}}]
	\label{thm_poly_phi_c}
	Let $\alpha = \{\alpha_1,\ldots,\alpha_r\}$ be a strictly increasing subsequence of  $\{1,\ldots,n\}$.
	For $n \ge 0$ the function 
	$$
	P_\alpha(n) \,:=\,
	\begin{cases}
		\psi_{\{1,\ldots,n\} \setminus \alpha} \quad \text{ if } \alpha \subset \{1,\ldots,n\} \\
		0 \qquad\qquad\;\; \text{ otherwise}
	\end{cases}
	$$
	is a polynomial in $n$ of degree $||\alpha|| = \alpha_1+\cdots+\alpha_r$.
\end{theorem}

Finally, we provide our proof for the polynomiality of $\phi(n,d)$.

\begin{proof}[Proof of \autoref{head_cor}]
	By using \autoref{thmB} (ii) and (\ref{eq_phi_as_multdeg}) we obtain the equation
	$$
	\phi(n,d) \,=\, \deg^{m-d,d-1}(\Gamma) \,=\, \sum_{j=0}^{d-1} {(-1)}^j \beta(n,d-1-j).
	$$
	Therefore, it suffices to show that 
	$$
	\beta(n,d) \,=\, \sum_{\substack{\alpha \subset \{ 1, \ldots, n \}\\ ||\alpha|| = d} 
	} \psi_\alpha \psi_{\alpha^c}
	$$
	in $n$ of degree $d$.
	Since $\psi_\alpha$ does not depend on $n$, the result follows directly from \autoref{thm_poly_phi_c}.
\end{proof}

\section*{Acknowledgments}

I am  very grateful to Bernd Sturmfels for suggesting me to work on this problem.
I thank Mateusz Micha\l{}ek and Tim Seynnaeve for helpful conversations.
I also thank the organizers of the Linear Spaces of Symmetric Matrices working group at MPI MiS Leipzig.

%\bibliography{references}

\end{document}